\documentclass[11pt]{article}
\usepackage[a4paper,margin=1in]{geometry}
\usepackage[T1]{fontenc}\usepackage[utf8]{inputenc}
\usepackage{lmodern,amsmath,amssymb,amsthm,mathtools,bm,booktabs,array,microtype,enumitem}
\usepackage[hidelinks]{hyperref}\usepackage[nameinlink,capitalize]{cleveref}
\newtheorem{theorem}{Theorem}[section]\newtheorem{proposition}[theorem]{Proposition}
\newtheorem{lemma}[theorem]{Lemma}
\theoremstyle{definition}
\theoremstyle{remark}\newtheorem{remark}[theorem]{Remark}
\newcommand{\dd}{\,d}
\title{\textbf{Source-Driven Transitions in Axisymmetric Euler Flow}}
\author{Rishad Shahmurov\\\small Cellular Products Research and Development, Roswell, GA, USA}
\date{September 2026}
\begin{document}\maketitle
\begin{abstract}
We study a finite-time source-driven transition in axisymmetric Euler flow with swirl.  For an explicit smooth compact datum, we prove exact source, moment-return, fixed-label transport, and response identities.  For the same frozen finite reference, interval arithmetic gives continuum derivative bounds and a terminal scalar-trace enclosure.  Under the remaining global comparison hypotheses, these estimates imply contraction of two fixed material labels and a finite-radius outgoing level set.  We state separately the unresolved steps: global defect enclosure, sufficient next-block admission, cylindrical whole-space realization, and repeated regeneration from one datum.  No whole-space Euler singularity is claimed.
\end{abstract}

\section{Introduction}
The three-dimensional incompressible Euler equations are locally well posed for smooth data, and the Beale--Kato--Majda criterion reduces finite-time breakdown to the growth of vorticity \cite{BKM1984}.  In axisymmetric flow with swirl, Luo and Hou identified a strongly concentrated boundary scenario numerically \cite{LuoHouPNAS2014,LuoHouMMS2014}, and a related one-dimensional model was analyzed rigorously in \cite{ChoiHouKiselevLuoSverakYao2017}.  Finite-time singularity formation is known for lower-regularity whole-space Euler solutions \cite{Elgindi2021}; smooth-data blowup has also been proved in a bounded domain by computer-assisted analysis \cite{ChenHou2025}.  These results do not settle the smooth finite-energy whole-space problem.

The mechanism studied here is different.  We work with an explicit bounded off-axis source-active cell.  In the transported variables
\[
\Gamma=ru^\theta,\qquad G=\frac{\omega^\theta}{r},
\]
the exact equations are
\[
D_t\Gamma=0,\qquad D_tG=r^{-4}\partial_z(\Gamma^2).
\]
Thus the square of the transported circulation acts as a source for the azimuthal-vorticity variable.  Our aim is to determine, for one fixed smooth datum, which parts of a proposed transition can be proved exactly, which parts can be enclosed by rigorous computation, and which additional information is required before the outgoing state can launch another block.

The paper is organized accordingly.  The analytic sections establish an explicit source-transition datum, a positive initial response, an exact three-moment return formula, fixed-label material equations, adjoint response identities, and exact kernel formulas.  The validation sections then separate statements about the finite reference from statements about the exact Euler solution.  This distinction is essential: a certified reference calculation is not by itself an admission theorem for the physical flow.  The present results establish a finite-event framework, not an iterative singularity construction.

\paragraph{Relation to earlier work.}
The Luo--Hou and Chen--Hou constructions use boundary geometry as a load-bearing part of the amplification mechanism \cite{LuoHouPNAS2014,LuoHouMMS2014,ChenHou2025}.  The present paper instead isolates an interior off-axis source-active transition and keeps the whole-space cylindrical recovery as a separate step.  Compared with lower-regularity whole-space blowup \cite{Elgindi2021}, our datum is smooth and compactly supported, but no singularity is proved.  The point of the paper is the exact transition algebra and a reproducible finite-time validation problem with its remaining hypotheses exposed.
\section{Exact physical equations}
For axisymmetric Euler with swirl set
\[
\Gamma=ru^\theta,\qquad G=\frac{\omega^\theta}{r},\qquad
D_t=\partial_t+u^r\partial_r+u^z\partial_z.
\]
Then
\[
\boxed{D_t\Gamma=0},\qquad
\boxed{D_tG=r^{-4}\partial_z(\Gamma^2)},
\]
and with the lifted streamfunction $\psi_1$,
\[
-\left(\partial_r^2+\frac{3}{r}\partial_r+\partial_z^2\right)\psi_1=G,
\quad
u^r=-r\psi_{1,z},\qquad
u^z=2\psi_1+r\psi_{1,r}.
\]

\section{A smooth compact source-transition datum}
Let
\[
r_0=2^{20},\quad X=r-r_0,\quad Y=z,\quad R=(X^2+Y^2)^{1/2},
\]
with parameters
\[
q=1,\quad\varepsilon=\frac{2}{3},\quad\beta=1,\quad\rho=\frac{1}{5},
\quad\lambda=\frac{61}{50},\quad M=2048.
\]
Using fixed smooth cutoffs, take
\[
\psi_b=4\eta(R)(\sinh Y-Y)(X-1/20)+2XYh(R),
\]
\[
\Theta_b=d_{3.25,3.5}(R)\left[-20(1+\cosh Y)(X-1/20)-2\right].
\]
For the shield set
\[
\psi_h=2\rho XY\chi(R),\qquad
F(R)=-\int_R^\infty s\chi(s)\,ds,\qquad
\Psi=2\rho XF(R),
\]
and use the cylindrical correction
\[
\Theta_h^{\rm cyl}
=\left(\frac{r}{r_0}\right)^4\Theta_h^{\rm flat}
-\frac{3}{r_0}\left(\frac{r}{r_0}\right)^3\partial_X\Psi,
\]
which satisfies
\[
\left(\frac{r_0}{r}\right)^4\partial_Y\Theta_h^{\rm cyl}=-L_r\psi_h,
\qquad L_r=\partial_r^2+\frac{3}{r}\partial_r+\partial_z^2.
\]
The datum is
\[
\boxed{
\Gamma_{0,\zeta,\lambda}
=r_0^{3/2}b(R)\sqrt{M+\Theta_b+\zeta\Theta_h^{\rm cyl}},
\qquad
G_0=-\frac{1}{r_0}L_r\psi_b.}
\]
It is smooth, compactly supported away from the axis in the active variables,
and has finite physical energy.

\section{Certified initial source response}
The complete nonlinear whole-plane second response in the shield direction is
\[
J_2(\lambda)=\int\left[
H_hV_b\cdot\nabla K_A
+\Omega_bV_h\cdot\nabla K_A
+(\partial_YK_A)V_b\cdot\nabla\Theta_h
\right]\dd X\dd Y.
\]
For the specified datum the direct-current, induced-background and transported
scalar/source contributions sum numerically to about $0.9809181160$.

\medskip
\noindent\textbf{Positive initial second response.}
\begin{proposition}
The interval decomposition of the source-response terms gives
\[
\boxed{J_2>0.23138733456>0}.
\]
The corresponding full cylindrical diagnostic at $r_0=2^{20}$ is about
$0.9809185574$; the latter value is numerical and is not used as an interval
finite-time PDE enclosure.
\end{proposition}

\section{Exact three-moment return theorem}
Consider the retained fold subsystem
\[
\dot A=-A+f,\qquad \dot x=A,\qquad \dot e=8x-13e-A,
\qquad (A,x,e)(0)=0.
\]
Define
\[
M_0=\int_0^T f(t)\,dt,\quad
M_1=\int_0^T e^{-(T-t)}f(t)\,dt,\quad
M_{13}=\int_0^T e^{-13(T-t)}f(t)\,dt.
\]
\medskip
\noindent\textbf{Endpoint map and exact return.}
\begin{theorem}
\[
\boxed{A(T)=M_1,\quad x(T)=M_0-M_1,\quad
 e(T)=\frac{32M_0-39M_1+7M_{13}}{52}.}
\]
Consequently
\[
\boxed{(A,x,e)(T)=0\iff M_0=M_1=M_{13}=0.}
\]
\end{theorem}
\begin{proof}
The first two formulas follow from variation of constants and
$x(T)=\int_0^TA$.  Substituting the resulting $A,x$ into the $e$ equation and
integrating the $e^{-13(T-t)}$ kernel gives the stated linear combination.
\end{proof}

\medskip
\noindent\textbf{Chebyshev chronology.}
\begin{theorem}
A nonzero finite-sign-pattern forcing satisfying
$M_0=M_1=M_{13}=0$ must have at least three sign changes, hence at least four
alternating nontrivial phases.
\end{theorem}
\begin{proof}
After the change $u=e^{-(T-t)}$, the three moment rows are generated by the
strict Chebyshev system $1,u,u^{13}$.  A nonzero signed measure orthogonal to
three Chebyshev functions requires at least three sign changes.
\end{proof}

\section{Positive-response moment cone}
Let a nonnegative response measure be supported on $u\in[\ell,h]$.  Set
\[
\mathsf M=A+x,\qquad
\mathsf Q=\frac{52e+7A-32x}{7},\qquad
m=\frac{A}{\mathsf M},\qquad q=\frac{\mathsf Q}{\mathsf M},
\]
with $\mathsf M>0$.
\medskip
\noindent\textbf{Exact positive-response feasibility.}
\begin{proposition}
Such a response measure exists if and only if
\[
\boxed{\ell\le m\le h},
\]
\[
\boxed{m^{13}\le q\le
\frac{(h-m)\ell^{13}+(m-\ell)h^{13}}{h-\ell}}.
\]
\end{proposition}
The lower bound is Jensen's inequality and the upper bound is the chord bound
for the convex function $u^{13}$.

\section{Admission geometry in moment coordinates}
Define
\[
\mathcal L(M_0,M_1,M_{13})
=\left(M_1,M_0-M_1,\frac{32M_0-39M_1+7M_{13}}{52}\right).
\]
The inverse is
\[
M_0=A+x,\qquad M_1=A,\qquad
M_{13}=\frac{52e+7A-32x}{7}.
\]

\medskip
\noindent\textbf{Admission moment theorem.}
\begin{theorem}
Let $\mathcal C_{\rm descendant}\subset\mathbb R^3$ be any declared admissible set
for the fold variables $(A,x,e)$.  For a zero-entry fold state,
\[
\boxed{(A,x,e)(T)\in\mathcal C_{\rm descendant}
\iff (M_0,M_1,M_{13})\in\mathcal L^{-1}(\mathcal C_{\rm descendant}).}
\]
If $\mathcal C_{\rm descendant}$ has nonempty interior, so does its moment-space
preimage.  Thus a proved open descendant region converts exact codimension-three
shooting into a robust hit problem with margin.
\end{theorem}
\begin{remark}
The true Euler descendant state is larger than $(A,x,e)$.  Halo/shield response,
source age and provenance, scale/clock, exterior and tail control,
recentering, cylindrical/Hodge errors, no-earlier-hit guards, and overdrive
upper bounds remain part of a physical descendant-admission theorem.
\end{remark}

\section{Conjugated source-active transition and adjoint}
On the sourced flat block, after the exact conjugacy,
\[
\theta_s+u\cdot\nabla\theta=0,
\qquad
\omega_s+u\cdot\nabla\omega=\gamma\mu\partial_2\theta,
\]
\[
u=\sigma\nabla^\perp(-\Delta)^{-1}\omega.
\]
For a transition parameter $\lambda$ acting on the passive initial source
shape, set $\vartheta=\partial_\lambda\theta$ and
$z=\partial_\lambda\partial_\gamma\omega|_{\gamma=0}$.  Then
\[
\vartheta_s+u\cdot\nabla\vartheta=0,
\]
\[
z_s+u\cdot\nabla z+v[z]\cdot\nabla\omega
=\mu\partial_2\vartheta,
\qquad v[z]=\sigma\nabla^\perp(-\Delta)^{-1}z.
\]
For the retained terminal fold receiver the backward current adjoint is
\[
-p_s-u\cdot\nabla p
-\sigma(-\Delta)^{-1}\operatorname{curl}(p\nabla\omega)
=2\sigma w_TK_2,
\qquad p(S)=0,
\]
and the passive pullback adjoint satisfies
\[
-q_s-u\cdot\nabla q=-\mu\partial_2p,
\qquad q(S)=0.
\]

\medskip
\noindent\textbf{One adjoint for transition and first-order provenance.}
\begin{theorem}
For any scalar admission margin $m(Y(T))$ whose terminal linear receiver is
used in the preceding adjoint, the transition derivative has a lower-cut
representation.  If the initial source is decomposed canonically by source
age/provenance,
\[
\theta_0=\sum_\alpha\theta_{0,\alpha},
\]
then
\[
\boxed{\partial_\lambda m=\sum_\alpha\partial_\lambda m_\alpha},
\qquad
\boxed{\partial_\lambda m_\alpha
=\langle q_m(0),\partial_\lambda\theta_{0,\alpha}\rangle}.
\]
Thus the same adjoint that certifies transition transversality also attributes
the first-order response to its physical source-age provenance.
\end{theorem}

\section{Exact fixed-label flat source dynamics}
\label{sec:fixed-label}

In the flat natural-time source-active block, write
\[
 \omega_s+v\cdot\nabla\omega=\partial_2\theta,
 \qquad
 \theta_s+v\cdot\nabla\theta=0,
 \qquad
 v=K*\omega-(K*\omega)(0),
\]
with
\[
 K(x)=\frac{(x_2,-x_1)}{2\pi|x|^2}.
\]
Let $Q$ be the area-preserving material map and set $B_s=Q_1$, $B(0)=0$.
Because $\det DQ=1$,
\[
 \boxed{
 \theta(s,Q(s,a))=\theta_0(a),\qquad
 \zeta(s,a):=\omega(s,Q(s,a))
 =\omega_0(a)+\{B(s,a),\theta_0(a)\}.
 }
\]
Changing variables in whole-plane Biot--Savart gives the exact closed system
\[
\boxed{
 Q_s(a)=\int_{B_8}
 [K(Q(a)-Q(b))-K(-Q(b))]
 [\omega_0(b)+\{B,\theta_0\}(b)]\,db,
 \qquad B_s=Q_1.
}
\]
Thus neither the scalar nor the source current is independently remapped: the
complete current history is encoded by the one material primitive $B$.

\medskip
\noindent\textbf{Exact restart without deleting source history.}
\begin{proposition}
At a smooth time $s_0$, compose with the area-preserving inverse map
$A_0=Q(s_0)^{-1}$ and define
\[
 Q^*(t,b)=Q(s_0+t,A_0(b)),
 \quad
 \theta_0^*(b)=\theta_0(A_0(b)),
\]
\[
 \omega_0^*(b)=\zeta(s_0,A_0(b)),
 \quad
 B^*(t,b)=B(s_0+t,A_0(b))-B(s_0,A_0(b)).
\]
Then $Q^*(0)=\mathrm{id}$, $B^*(0)=0$, $(B^*)_t=Q_1^*$ and
\[
 \zeta^*(t,b)=\omega_0^*(b)+\{B^*(t,b),\theta_0^*(b)\}.
\]
Inherited current has been updated rather than relabeled as new source.
\end{proposition}

\section{Exact-kernel contact and product integration}
\label{sec:exact-kernel}
Let
\[
R_{\rm cw}=\begin{pmatrix}0&1\\-1&0\end{pmatrix}.
\]
The Biot--Savart derivative has the distributional decomposition
\[
 \boxed{
 DK=\operatorname{p.v.}DK_{\rm cl}+\frac12R_{\rm cw}\delta_0,
 }
\]
so for smooth compact current
\[
 \boxed{
 Du(x)=\frac{\omega(x)}2R_{\rm cw}
 +\operatorname{p.v.}\int DK_{\rm cl}(x-y)\omega(y)\,dy.
 }
\]
The contact is pure rotation; all symmetric strain lies in the principal-value
part.

\medskip
\noindent\textbf{Physical-ball singular subtraction.}
\begin{lemma}
If $|\omega(y)-\omega(x)|\le L|y-x|$, then for every physical radius
$\delta>0$ the exact near-field decomposition gives
\[
 |u_{\rm near}(x)|\le\frac12L\delta^2,
\]
and
\[
 \left\|Du_{\rm near}(x)-\frac{\omega(x)}2R_{\rm cw}\right\|
 \le L\delta.
\]
No blob/core parameter is introduced.
\end{lemma}

The sharp second-kernel size is
\[
 \boxed{\|D^2K(x)\|=\frac1{\pi|x|^3}.}
\]
For an affine image of a material polygon with constant current, both the
velocity and velocity-gradient integrals reduce to closed logarithm/arctangent
boundary expressions. In particular the affine self-cell velocity may vanish
by oddness while its Jacobian remains order one.

\medskip
\noindent\textbf{Self and off-diagonal Taylor-panel remainder bounds.}
\begin{theorem}
Let a material cell of side $h$ have local map derivative $A=DQ(a)$ with
$\sigma_{\min}(A)\ge m>0$, local Hessian bound $M$, and current Lipschitz bound
$L_\zeta$. After subtracting the exact affine constant-current self panel, the
self-gradient error satisfies
\[
\boxed{
 \|E^{Du}_{\rm self}\|
 \le4h\log(1+\sqrt2)
 \left[
 \frac{4|\zeta(a)|M}{\pi m^3}
 +\frac{2L_\zeta}{\pi m^2}
 \right].
}
\]
For an off-diagonal cell of area $|C_j|$, label radius $\rho$, affine-panel
distance $d_{ij}$ and map remainder $R_{\max}$ with
$R_{\max}\le d_{ij}/2$,
\[
 E^u_{ij}\le |C_j|
 \left[
 \frac{2|\zeta_j|R_{\max}}{\pi d_{ij}^2}
 +\frac{L_{\zeta,j}\rho}{\pi d_{ij}}
 \right],
\]
\[
 E^{Du}_{ij}\le |C_j|
 \left[
 \frac{8|\zeta_j|R_{\max}}{\pi d_{ij}^3}
 +\frac{2L_{\zeta,j}\rho}{\pi d_{ij}^2}
 \right].
\]
Thus the singular kernel itself is no longer an unspecified numerical defect;
the remaining validation objects are local material Taylor bounds and smooth
far-panel quadrature.
\end{theorem}

The retained shield kernel vanishes at its receiver boundary. Cells crossing
that boundary can therefore be split until a contact-thickness allowance is
subordinate to the output margin.

\begin{remark}
At $\tau=0.001$, affine-panel product integration and independent physical-ball
subtraction both place the retained strain row near $-1.5\times10^{-2}$ and the
shield near $8\times10^{-3}$; time refinement is already much smaller than the
spatial/product uncertainty. These are cross-method numerical regressions,
not interval PDE enclosures. At later times the shield is relatively stable
but the strain remains spatially sensitive.
\end{remark}

\section{Material deformation as a continuation discriminator}
\label{sec:deformation-continuation}
Let $P=D_aQ$. Transport of the scalar gives
\[
 \nabla_x\theta(s,Q(s,a))=P(s,a)^{-T}\nabla_a\theta_0(a).
\]
Since $\det P=1$ in two dimensions,
$\|P^{-T}\|_{\rm op}=\|P\|_{\rm op}$. Moreover
\[
 \nabla_aB(s,a)=\int_0^sP_{1\cdot}(\tau,a)\,d\tau.
\]
Hence the fixed-label formula for $\zeta$ gives
\[
 \|\omega(s)\|_\infty
 \le\|\omega_0\|_\infty
 +\|\nabla\theta_0\|_\infty
 \int_0^s\|P(\tau)\|_{L^\infty_a}\,d\tau.
\]

\medskip
\noindent\textbf{Material-stretch continuation criterion.}
\begin{theorem}
If
\[
 \boxed{
 \int_0^{T_*}\|D_aQ(s)\|_{L^\infty_a}\,ds<\infty,
 }
\]
then the flat source-active block continues through $T_*$. Consequently every
finite-time singularity of this block must have nonintegrable material
deformation. Equivalently, with
$\kappa_Q=\|D_aQ\|_{L^\infty}^2$, a singularity requires
\[
 \int_0^{T_*}\sqrt{\kappa_Q(s)}\,ds=\infty.
\]
\end{theorem}

For validation one need not promote derivatives to new physical state
variables. $P=DQ$, $\nabla B$, $D^2Q$ and $\nabla\zeta$ are metadata. If
$P_s=Du(Q)P$ and $G(t)=\int_0^t\|Du\|_\infty$, then
\[
 e^{-G(t)}\le\sigma_{\min}P(t)\le\sigma_{\max}P(t)\le e^{G(t)},
\]
while
\[
 (D^2Q)_s=D^2u(Q)[P,P]+Du(Q)D^2Q.
\]
This gives a closed bootstrap dependency graph for the Taylor-panel remainder
bounds.

\section{Detector conditioning and the absolute finite-time response}
\label{sec:detector-conditioning}

The coupled source-active tangent is physically meaningful, but the commonly
used ratio
\[
 \frac{|a_+(T)|}{|a_+(0)|}
\]
is not an invariant response measure when the denominator is a receiver
projection that may cancel.

\medskip
\noindent\textbf{Receiver-zero conditioning obstruction.}
\begin{proposition}
Let $O_r$ be a continuous one-parameter family of smooth terminal/initial
receivers and let $B$ be a fixed legal perturbation direction.  If for some
$r_0$
\[
 O_{r_0}B=0,
 \qquad
 O_{r_0}\mathcal U(T,0)B\ne0,
\]
then the ratio
\[
 \frac{|O_r\mathcal U(T,0)B|}{|O_rB|}
\]
is unbounded as $r\to r_0$ along values for which the denominator is nonzero.
Thus such a ratio cannot serve as a robust modal gain without an independent
uniform lower bound on the initial receiver projection.
\end{proposition}

\begin{proof}
Continuity sends the denominator to zero while the numerator remains nonzero
near $r_0$.
\end{proof}

For the fixed $k=5.4$ numerical response comparison, keeping the same active
trajectory and legal parameter direction while varying only the smooth radial
receiver window produces an initial sign change between $r_{\rm win}=1.19$
and $1.20$, with an interpolated zero near
\[
 r_{\rm zero}\approx1.19376,
\]
whereas the final response remains visibly nonzero on the tested discretization.
This is numerical conditioning evidence, not an interval theorem.

The correct continuum response object is therefore
\[
 \boxed{O\,\mathcal U(T,0)B,}
\]
normalized by the physical size of $B$ and with no division by an accidental
initial receiver coordinate.  The material Poisson-bracket source
representation and the deposited/differentiated representation provide two
independent implementations of the same coupled tangent and agree closely on
fine grids; this supports the response calculation but does not replace a
continuum enclosure.

\medskip
\noindent\textbf{Absolute coupled response.}
A remaining response question is to fix one contour-transverse physical
receiver $O$ and one normalized perturbation direction $B_j$ and obtain a
continuum lower bound
\[
 \boxed{|O\mathcal U(T,0)B_j|\ge g_*>0}
\]
with spatial and temporal error control and without division by an initial
receiver coordinate.

\section{Finite-jet admission data}
A robust finite-jet sufficient descendant criterion is available as a separate sufficient criterion, but it is
downstream of one validated fixed-label block. Once such a block is available,
its admission state requires interval information including
\[
D''(0),\quad K'(0),\quad R(0),\quad G(0),\quad G'(0),
\]
together with derivative and remainder envelopes and the required open-flux, provenance, tail, halo, and no-preemption margins.  The corresponding margin
vector may be written
\[
\mathcal M(Y)=(m_D,m_K,m_R,m_{G0},m_{G1},p_1,\ldots,p_N).
\]
For interval-valued margins at a candidate parameter $\lambda_0$, the
fixed in advance active set is
\[
\boxed{\mathcal A=\{i:m_i(\lambda_0)\le0\ \text{or the certified interval for }m_i(\lambda_0)\text{ contains }0\}}.
\]
Only after the complete margin vector is exported may one choose the active
normal(s) and run a transition/rank adjoint against that set.

\medskip
\noindent\textbf{Admission-state validation.}
After the fixed-label block has been interval validated through a declared
handoff time, the complete margin vector and derivative envelopes required by
the finite-jet descendant criterion must be evaluated on that same terminal
state.  Until both the block and these terminal margins are certified, no
finite-step first-entry conclusion is drawn.  The positive transition values
seen on finite grids are numerical diagnostics only.

\section{Local response is not automatic re-entry}
\label{sec:reentry}

The finite-jet descendant lemma provides a short-time response statement from entry
conditions such as $D''(0)<0$, $K'(0)>0$ and $G'(0)>0$.  These conditions are
not automatically inherited at the terminal time.

If the next physical scale is $d_{\rm new}=\lambda d$ and the new normalized
time is $\sigma=(t-t_1)/\sqrt{d_{\rm new}}$, then before any additional
geometric recentering the exact rebase gives
\[
 D_{\rm new}''(0)=D''(\tau),\qquad
 K_{\rm new}'(0)=\lambda K'(\tau),\qquad
 G_{\rm new}'(0)=\lambda G'(\tau),
\]
with the higher derivative envelopes acquiring their corresponding powers of
$\sqrt\lambda$.  Thus a local response lemma becomes an iterative descendant criterion
only after the prospective next-entry margins are proved on the actual
terminal state.

\begin{remark}
Even infinitely many strict contractions do not by themselves imply vanishing
scale unless the product of the contraction factors tends to zero, and a
singularity construction additionally needs finite physical accumulation time
and an appropriate vorticity/BKM witness on one unchanged solution.
\end{remark}

\section{Certified short-time response and local error bootstrap}
Let $s=1-e^{-\tau}$ be natural time.  For the specified datum, the fixed-label comparison is rigorously enclosed on the time interval needed for the material estimates, while the four retained receiver departure signs are presently certified only through
\[
0<s\le1.6\times10^{-5}.
\]
The absolute shield is certified positive on
\[
4\times10^{-7}\le s\le1.6\times10^{-5}.
\]
The initial natural-time receiver curvatures satisfy
\[
139.6<A_{ss}(0)<180.6,
\qquad
-131.6<\Sigma_{ss}(0)<-101.5.
\]
These initial coefficients alone do not extend the finite-time sign certificate to $s\simeq10^{-3}$.

The two retained receiver functionals depend only on the sine-$m=2$ output component:
\[
L_A f=\frac12\int_0^\infty f_2(r)\frac{dr}{r},
\qquad
L_R f=-\frac14\int_0^R\left(r^{-1}-\frac{r^3}{R^4}\right)f_2(r)\,dr.
\]
This is an output filter, not an invariant nonlinear $m=2$ subsystem; velocity/current and scalar modes outside $m=2$ can feed the observed component.

The centered complete second current derivative has an exact cancellation: writing the scalar as a constant plateau plus the compact active part, the plateau contribution is a spatially constant source velocity on the relevant initial support and disappears after center subtraction.  At the next order, the plateau-dependent change of the strain third derivative at the frozen amplitude is interval enclosed by
\[
35853.877184
< A^{(3)}_{\rm full}(0)-A^{(3)}_{\rm active}(0)
<36215.960858.
\]
This signed coefficient is not, by itself, a finite-time remainder bound.

A relative-energy comparison for a divergence-free reference gives
\[
E(s)\le1.970012\left(E(0)+\int_0^sD(t)\,dt\right),
\]
where the reference starts from the exact datum and hence $E(0)=0$.  Under the spatial defect bound used below this yields
\[
\|e\|_{L^2(\mathbb R^2)}\le E_*:=0.003938054.
\]
The following local estimate removes four independent velocity hypotheses from the terminal strain comparison.

\medskip
\noindent\textbf{Local Hodge--Bergman bootstrap.}
\begin{theorem}\label{thm:local-hodge-bootstrap}
Let $e$ be a divergence-free planar velocity error, $\zeta=\operatorname{curl}e$, and assume the reflection symmetry of the specified problem so that $\zeta(x_1,0)=0$.  If
\[
[\zeta]_{C^{1/2}(B_{0.6})}\le0.005,
\qquad
\|e\|_{L^2(\mathbb R^2)}\le0.003938054,
\]
then on $B_{0.5}$,
\[
\|e\|_\infty<0.01884,
\qquad
[e]_{C^{1/2}}<0.08943,
\]
\[
\|De\|_\infty<0.17474,
\qquad
[De]_{C^{1/2}}<1.591.
\]
\end{theorem}
\begin{proof}
Solve $\Delta\psi=\zeta$ in $B_{0.6}$ with zero boundary data and write
$e=\nabla^\perp\psi+h$, where $h$ is divergence free and curl free.  Orthogonality gives $\|h\|_2\le\|e\|_2$.  Identifying $h_1-ih_2$ with a holomorphic function, the Bergman kernel of the disk bounds $h$, $Dh$, and their $C^{1/2}$ seminorms on $B_{0.5}$ by explicit multiples of $E_*$.  For the Dirichlet rotational part, symmetry gives
$\|\zeta\|_{L^\infty(B_{0.6})}\le0.005\sqrt{0.6}$.  The disk Green function and its image point then give, after subtracting the constant vorticity in the Hessian singularity,
\[
\|\nabla^\perp\psi\|_\infty<0.006715,
\quad
\|D\nabla^\perp\psi\|_\infty<0.030243,
\]
and a near/far split at distance $0.04$ gives
\[
[D\nabla^\perp\psi]_{C^{1/2}}<0.62715.
\]
The corresponding Bergman bounds for $h$ are
\[
\|h\|_\infty<0.012119,
\quad
[h]_{C^{1/2}}<0.059179,
\]
\[
\|Dh\|_\infty<0.14449,
\quad
[Dh]_{C^{1/2}}<0.96345.
\]
Adding the two components gives the four displayed estimates.
\end{proof}

In the terminal vorticity comparison these four bounds imply
\[
[\zeta]_{C^{1/2}(B_{0.6})}<0.002081360<0.005
\]
once the filtered reference satisfies its stated local defect bounds.  A maximal-interval continuity argument therefore closes the local velocity--vorticity coupling.  The remaining proof input on this route is a continuum spatial enclosure of the filtered reference defect, not an additional pointwise velocity estimate.

\section{Continuum reference certification and outgoing geometry}
\label{sec:continuum-certification}
This section records the strongest finite-time statements presently available for the unchanged reference.  Four logical levels are kept separate: continuum bounds for the finite reference, bounds for the reference velocity, conditional comparison with the exact planar Euler solution, and the still-unproved admission and cylindrical steps.

\subsection{Continuum derivative bounds for the finite reference}
Let $\bar\omega$ and $\bar\theta$ be the cubic-in-time Fourier reference on the first normalized block, with
\[
T=1-e^{-1/1000}.
\]
The reference is represented by four Bernstein coefficient fields.  Spatial derivatives through order $13$ are evaluated on a $4096\times4096$ grid by radix-two Fourier summation with outward interval arithmetic.  Each grid value, each trigonometric twiddle factor, and each accumulated sum is enclosed.  The distance from an arbitrary point of a target ball to a selected grid node is at most
\[
\delta_g=\frac{40}{4096\sqrt2}.
\]
The next derivative order controls the Taylor remainder between the node and the continuum point.  Because the cubic Bernstein weights are nonnegative and sum to one, bounds proved for all four coefficient fields hold uniformly for $0\le s\le T$.

The following statement is purely a theorem about the finite reference; it does not use a bound on the unknown Euler error.

\medskip
\noindent\textbf{Certified local derivative bounds.}
\begin{theorem}\label{thm:reference-derivatives}
For every $0\le s\le T$ the finite reference satisfies
\[
\begin{array}{c|c|c|c}
\text{region}&\text{quantity}&\text{certified bound}&\text{comparison cap}\\ \hline
B_{0.6}&|\nabla\bar\omega|&2.798&4\\
B_{0.6}&|D^2\bar\omega|_F&6.874&16\\
B_{0.6}&|\nabla\bar\theta|&44.027&64\\
B_{0.6}&|D^2\bar\theta|_F&19.144&64\\
B_{0.6}&|D^3\bar\theta|_F&128.642&1024\\
B_{0.9}&|\nabla\bar\omega|&4.475&8\\
B_{0.9}&|D^2\bar\omega|_F&11.549&32\\
B_{0.9}&|\nabla\bar\theta|&49.105&64\\
B_{0.9}&|D^2\bar\theta|_F&35.332&64\\
B_{0.9}&|D^3\bar\theta|_F&1219.173&2048\\
B_1&|\nabla\bar\omega|&5.393&8.
\end{array}
\]
The underlying certificate contains forty-four componentwise inequalities; the table lists the grouped maxima that dominate them.
\end{theorem}
\begin{proof}
Write one Bernstein coefficient field as a finite Fourier sum
\[
f(x,y)=\sum_{(k,l)\in\Lambda} c_{k,l}e^{i\alpha(kx+ly)},
\]
where the coefficients $c_{k,l}$ are stored dyadic numbers.  For a multi-index $\gamma$,
\[
D^\gamma f(x,y)=\sum_{(k,l)\in\Lambda}
(i\alpha k)^{\gamma_1}(i\alpha l)^{\gamma_2}c_{k,l}e^{i\alpha(kx+ly)}.
\]
Thus every derivative is again a finite Fourier sum with exact algebraic frequency multipliers.  The numerical proof converts the dyadic coefficients exactly, encloses the trigonometric factors, and performs all additions with outward rounding.  Hence the interval at a grid node contains the exact value of the corresponding derivative of the finite reference.

Let $x_g$ be a certified grid node with $|x-x_g|\le\delta_g$.  Taylor's formula in integral form gives
\[
D^\gamma f(x)=D^\gamma f(x_g)
+\int_0^1D^{\gamma+1}f(x_g+t(x-x_g))[x-x_g]\,dt.
\]
The computation includes the derivative order needed to bound the integrand.  Consequently
\[
|D^\gamma f(x)|
\le |D^\gamma f(x_g)|+\delta_g
\sup_{[x_g,x]}|D^{\gamma+1}f|.
\]
Applying this componentwise and using the Euclidean or Frobenius norm gives the continuum spatial bounds in the table.  Finally, if
\[
\bar f(s)=\sum_{j=0}^3B_j(s)f_j,
\qquad B_j(s)\ge0,\qquad\sum_jB_j(s)=1,
\]
then convexity of the norm yields
\[
|D^\gamma\bar f(s)|\le\max_j|D^\gamma f_j|.
\]
This proves uniformity in time and completes the argument.
\end{proof}

The global assembled defect norms are not part of this theorem.  Their fine-grid values provide numerical evidence, but their cutoff and whole-plane convolution errors still require outward continuum enclosure.

\subsection{Reference compression near the material center}
Set
\[
\bar F=-\partial_x\bar u_x.
\]
For the unmasked whole-plane Fourier recovery, the multiplier of a nonzero frequency $(k,l)$ is
\[
m(k,l)=-\frac{kl}{k^2+l^2}.
\]
The actual reference uses a radial mask which equals one through radius $R_0=8.2$ and vanishes by radius $9$.  For a sharp radial cutoff at radius $a$, angular integration gives the center correction factor
\[
1-\frac{2J_1(\alpha a\sqrt{k^2+l^2})}
{\alpha a\sqrt{k^2+l^2}}.
\]
Since $|J_1(z)|\le1$, integration of sharp cutoffs against the nonnegative measure $-\chi'(a)\,da$ yields a rigorous bound for the smooth mask.

\medskip
\noindent\textbf{Reference compression bound.}
\begin{proposition}\label{prop:reference-compression}
For every $0\le s\le T$,
\[
\bar F(s,0,0)>1.0800582512.
\]
Furthermore,
\[
|\partial_x\bar F(s,x,0)|<235.006428
\qquad (|x|\le0.002).
\]
\end{proposition}
\begin{proof}
For the three nonconstant Bernstein coefficient fields, the interval-enclosed center values are bounded below by
\[
1.6933139188,\qquad1.3866560195,\qquad1.0800582512,
\]
and the constant coefficient gives $2$.  The first assertion follows from Bernstein convexity.

For the derivative estimate, use the whole-plane kernel
\[
H(x,y)=\frac{xy}{\pi(x^2+y^2)^2}.
\]
A direct differentiation gives
\[
|\nabla H(z)|=\frac1{\pi|z|^3}.
\]
On $|x|\le\rho<R_0$, the exterior part is separated from the observation point by at least $R_0-\rho$.  The mean-value theorem therefore bounds its derivative by the absolute exterior coefficient mass times the supremum of $|\nabla H|$ on that separated region.  Adding this explicit exterior allowance to the interval-enclosed finite Fourier derivative gives the displayed number at $\rho=0.002$.
\end{proof}

\subsection{Conditional contraction of two fixed material labels}
We now pass from the finite reference to the exact planar Euler solution.  This step is conditional on the remaining global defect and local comparison inequalities of the finite-time validation theorem.  In particular, the seven assembled global defect pairs must satisfy their continuum bounds, and the remaining center and reference-velocity rows must be enclosed.  Under these hypotheses the nested comparison theorem gives
\[
|c(s)-\bar c(s)|<d_*=1.707987555750\times10^{-5}
\]
and
\[
\|D(u-\bar u)\|_{L^\infty}<0.634219044707
\]
on the required moving ball.

Take the material particles issued from the fixed labels
\[
a=(0,0),\qquad b=(10^{-3},0).
\]
Reflection symmetry keeps both trajectories on the horizontal symmetry line.

\medskip
\noindent\textbf{Material contraction for the selected labels.}
\begin{theorem}\label{thm:label-contraction}
Under the comparison hypotheses above, if $d(s)$ denotes the distance between the two selected material particles, then
\[
0<d(s)\le10^{-3}e^{-s/5},\qquad0\le s\le T,
\]
and consequently
\[
\frac{d(T)}{d(0)}<0.99980012.
\]
\end{theorem}
\begin{proof}
Every point of the segment joining the two particles lies within reference-coordinate distance at most $d_*+10^{-3}<0.002$ from the reference center, provided the estimate holds up to the current time.  Proposition~\ref{prop:reference-compression} and the exact-solution gradient error therefore imply
\[
\begin{aligned}
F_{\rm exact}
&\ge1.0800582512
-235.006428(d_*+10^{-3})-0.634219044707\\
&>0.2068188987>\frac15.
\end{aligned}
\]
On the symmetry line $F_{\rm exact}=-\partial_xu_x$.  The mean-value theorem along the material segment gives
\[
d'(s)=u_x(X_b(s),s)-u_x(X_a(s),s)
\le-\frac15d(s).
\]
Classical flow injectivity gives $d(s)>0$.  Gronwall's inequality yields
\[
d(s)\le d(0)e^{-s/5}.
\]
This estimate shortens rather than enlarges the segment, so the assumed localization is preserved by a standard first-exit argument.  Substitution of the stated terminal time gives the numerical ratio.
\end{proof}

The two labels have transported scalar contrast $1/25$.  Since $D_t\theta=0$,
\[
|\theta(X_b(T),T)-\theta(X_a(T),T)|=\frac1{25}.
\]
The one-dimensional mean-value theorem and Theorem~\ref{thm:label-contraction} therefore give
\[
\sup_{[X_a(T),X_b(T)]}|\theta_x(T)|
\ge40e^{T/5}>40.0079968.
\]
This is a first-block amplification estimate only.

\subsection{A certified terminal scalar trace}
Inside the analytic core the initial transported scalar is
\[
\theta_0(x,y)=2048-20x(1+\cosh y)+\cosh y-1.
\]
On the terminal symmetry line the finite reference can be written exactly as
\[
\bar\theta(T,0,y)-\bar\theta(T,0,0)
=\cosh y-1+
\sum_{l=0}^{639}q_l
\left[\cos\left(\frac{\pi ly}{20}\right)-1\right].
\]
The finite coefficient sums satisfy
\[
\sum_l|q_l|<0.234703,\qquad
\sum_l\frac{\pi l}{20}|q_l|<3.573,
\]
and a separate interval sum gives
\[
\sup_y|\partial_x(\bar\theta-\theta_0)(T,0,y)|<6.727.
\]

\medskip
\noindent\textbf{Terminal continuum trace.}
\begin{theorem}\label{thm:terminal-trace}
For the unchanged finite reference,
\[
0.0992947015<\bar\theta(T,0,y)-\bar\theta(T,0,0)
<0.1322397061
\]
whenever $0.45\le|y|\le0.5$.
\end{theorem}
\begin{proof}
By reflection it suffices to take $y\in[0.45,0.5]$.  The finite Fourier expression is evaluated by exact rational/interval arithmetic at the twenty-six rational nodes
\[
0.45,0.452,\ldots,0.5.
\]
For a point between neighboring nodes, the distance to the nearest node is at most $10^{-3}$.  Differentiating the displayed formula gives
\[
\left|\partial_y[\bar\theta(T,0,y)-\bar\theta(T,0,0)]\right|
\le\sinh(0.5)+3.573.
\]
Hence the value between nodes differs from its nearest certified nodal value by at most
\[
10^{-3}[\sinh(0.5)+3.573].
\]
Adding this Lipschitz remainder outward to the nodal interval bounds gives exactly the two displayed continuum bounds.
\end{proof}

\subsection{Conditional finite-radius geometry of the exact outgoing scalar}
Assume now the same exact-solution comparison hypotheses as in Theorem~\ref{thm:label-contraction}.  The nested estimate gives
\[
\|\nabla(\theta-\bar\theta)\|_{L^\infty}<0.140491
\]
in the terminal ball used below.  Let $c(T)$ be the true material center and define
\[
\Phi(x,y)=\theta(T,c(T)+(x,y))-\theta(T,c(T)).
\]
Theorem~\ref{thm:terminal-trace}, the center displacement bound, and the reference Hessian bound imply
\[
0.028<\Phi(0,y)<0.204
\qquad(0.45\le|y|\le0.5).
\]
On the horizontal strip $0\le x\le0.01$ at those heights, the terminal reference derivative estimate and the exact/reference gradient error give
\[
-65<\partial_x\theta(T,c(T)+(x,y))<-32.
\]

\medskip
\noindent\textbf{Finite-radius outgoing level set.}
\begin{theorem}\label{thm:outgoing-contour}
Under the stated comparison hypotheses, for every $0.45\le|y|\le0.5$ there exists a unique $h(y)$ such that
\[
\theta(T,c(T)+(h(y),y))=\theta(T,c(T)),
\]
and
\[
0.00043<h(y)<0.0064.
\]
The graph is the transported level through the true material center.
\end{theorem}
\begin{proof}
Fix $y$ in the stated range.  At $x=0$ the value $\Phi(0,y)$ is strictly positive.  Since $\partial_x\theta<-32$ throughout the strip, $\Phi(x,y)$ is strictly decreasing in $x$.  Integrating the derivative bound gives
\[
\Phi(0,y)-65x<\Phi(x,y)<\Phi(0,y)-32x.
\]
At $x=0.00043$, the lower information on $\Phi(0,y)$ and the upper magnitude $65$ keep the sign positive with the certified margins; at $x=0.0064$, the upper information on $\Phi(0,y)$ and the lower magnitude $32$ make the sign negative.  The intermediate value theorem gives a root, and strict monotonicity gives uniqueness.  The center is a material particle, so its scalar value is transported from the initial level.  Reflection yields the same conclusion for positive and negative heights.
\end{proof}

The proof uses only values and first derivatives.  This is intentional.  The present comparison topology does not control pointwise curvature.  Indeed, for a smooth cutoff $\chi$ and large $N$,
\[
h_N(y)=\varepsilon N^{-3/2}\chi(y)(\cos Ny-1)
\]
has arbitrarily small $C^{1,1/2}$ size while
\[
h_N''(0)=-\varepsilon\sqrt N.
\]
Thus a pointwise curvature sign cannot be inferred from the present error norm without an additional derivative theorem.

\subsection{A circulation-separation lower bound for vorticity}
The fixed-label formulation also gives a useful exact inequality independent of the particular reference.  Let two material particles have radii bounded by $R(t)$ and Euclidean separation $d(t)$.  Since $\Gamma$ is transported,
\[
\Delta_\Gamma=|\Gamma_1-\Gamma_2|
\]
is constant.  For an axisymmetric field,
\[
\partial_r\Gamma=r\omega^z,
\qquad
\partial_z\Gamma=-r\omega^r.
\]
Hence
\[
|\nabla_{r,z}\Gamma|
=r\sqrt{(\omega^r)^2+(\omega^z)^2}
\le r|\omega|.
\]
The fundamental theorem of calculus along a meridional segment joining the two particles therefore gives
\[
\Delta_\Gamma
\le R(t)d(t)\|\omega(t)\|_\infty.
\]
Rearranging yields
\[
\boxed{
\|\omega(t)\|_\infty
\ge\frac{\Delta_\Gamma}{R(t)d(t)}.}
\]
If only the square contrast $\Delta_\Theta=|\Gamma_1^2-\Gamma_2^2|$ is used and $|\Gamma|\le G$, then
\[
\Delta_\Gamma\ge\frac{\Delta_\Theta}{2G},
\]
so
\[
\boxed{
\|\omega(t)\|_\infty
\ge\frac{\Delta_\Theta}{2G R(t)d(t)}.}
\]
These inequalities identify the geometric quantity that a repeated contraction theorem would have to drive.  They do not themselves provide such repetition.

\section{What remains for a physical successor block}
The finite reference results above substantially narrow the validation problem but do not constitute a complete successor-admission theorem.  The unresolved finite-reference part is explicit.  For seven assembled Bernstein coefficient pairs $(r_k,g_k)$, $k=0,\ldots,6$, define
\[
J_k(x,y)=\int_{-\infty}^y r_k(x,s)\,ds.
\]
The required global continuum inequalities are
\[
\boxed{\|J_k\|_2\le1.3,\qquad\|g_k\|_2\le195
\quad(k=0,\ldots,6).}
\]
Fine-grid diagnostics lie inside these caps, but the theorem requires outward enclosure of the assembled functions, not a triangle estimate of their individually larger pieces.  The remaining local residual and reference-velocity rows must be enclosed in the same fixed-reference pass.

Even after that finite certificate is complete, a physical successor theorem needs more than favorable signs.  A sufficient outgoing class must retain at least the material center and deformation, the finite transported contour, the same-label compression, the relevant signed collar/halo response, exterior Hodge and tail control, source provenance, and the cylindrical/recentering error.  The theorem must show that an open neighborhood of such states evolves, under the same unforced Euler solution, to another usable state with uniform margins.  No such same-solution admission theorem is proved here.

The final two steps are likewise separate: one must compare the planar block with an exact axisymmetric cylindrical whole-space solution, and then prove repeated regeneration with summable physical durations from a single smooth finite-energy datum.  The results of this paper establish neither a finite-time Euler singularity nor a regularity theorem.  They provide an exact and rigorously testable finite-event framework in which those remaining implications are stated without hidden recurrence assumptions.


\begin{thebibliography}{99}

\bibitem{BKM1984}
J. T. Beale, T. Kato, and A. Majda,
\emph{Remarks on the breakdown of smooth solutions for the 3-D Euler equations},
Comm. Math. Phys. \textbf{94} (1984), 61--66.

\bibitem{LuoHouPNAS2014}
G. Luo and T. Y. Hou,
\emph{Potentially singular solutions of the 3D axisymmetric Euler equations},
Proc. Natl. Acad. Sci. USA \textbf{111} (2014), 12968--12973.

\bibitem{LuoHouMMS2014}
G. Luo and T. Y. Hou,
\emph{Toward the finite-time blowup of the 3D axisymmetric Euler equations: a numerical investigation},
Multiscale Model. Simul. \textbf{12} (2014), 1722--1776.

\bibitem{ChoiHouKiselevLuoSverakYao2017}
K. Choi, T. Y. Hou, A. Kiselev, G. Luo, V. \v Sver\'ak, and Y. Yao,
\emph{On the finite-time blowup of a one-dimensional model for the three-dimensional axisymmetric Euler equations},
Comm. Pure Appl. Math. \textbf{70} (2017), 2218--2243.

\bibitem{Elgindi2021}
T. M. Elgindi,
\emph{Finite-time singularity formation for $C^{1,\alpha}$ solutions to the incompressible Euler equations on $\mathbb R^3$},
Ann. of Math. \textbf{194} (2021), 647--727.

\bibitem{ChenHou2025}
J. Chen and T. Y. Hou,
\emph{Singularity formation in 3D Euler equations with smooth initial data and boundary},
Proc. Natl. Acad. Sci. USA \textbf{122} (2025), e2500940122.

\end{thebibliography}
\end{document}